\documentclass[10pt]{amsart}

\usepackage[utf8]{inputenc}
\usepackage{amsmath}
\usepackage{amssymb}
\usepackage{amsfonts}
\usepackage{amsthm}
\usepackage{thmtools}
\usepackage{enumerate}
\usepackage[top=3cm, bottom=3cm, left=3cm, right=3cm]{geometry}
\usepackage{graphicx}
\usepackage[all]{xy}
\usepackage{titlesec}
\usepackage{hyperref}
\usepackage{rotating}
\usepackage{xcolor}
\usepackage{setspace}
\usepackage[T1]{fontenc}
\usepackage{tikz}
\usepackage{comment}

\titleformat{\section}
{\filcenter\large\bf} 
{\thesection.\ } 
{1pt} 
{} %

\titleformat{\subsection}[hang]
{\filcenter\bf}
{\thesubsection.}
{1pt}
{}

\declaretheoremstyle[bodyfont=\normalfont]{normalbody}

\declaretheorem[numberwithin=section,name=Theorem]{theorem}
\declaretheorem[sibling=theorem,style=normalbody,name=Definition]{definition}

\declaretheorem[sibling=theorem,name=Lemma]{lemma}
\declaretheorem[sibling=theorem,name=Proposition]{proposition}

\declaretheorem[sibling=theorem,style=normalbody,name=Remark]{remark}

\newcommand{\Z}{\mathbb{Z}}

\newcommand{\C}{\mathbb{C}}

\renewcommand{\P}{\mathbb{P}}

\newcommand{\Aut}{\operatorname{Aut}}

\newcommand{\GL}{\mathrm{GL}}
\newcommand{\PGL}{\mathrm{PGL}}
\newcommand{\Div}{\mathrm{Div}}

\newcommand{\Fix}{\mathrm{Fix}}
\newcommand{\id}{\mathrm{id}}
\newcommand{\SL}{\mathrm{SL}}

\DeclareFontFamily{U}{wncy}{}
\DeclareFontShape{U}{wncy}{m}{n}{<->wncyr10}{}
\DeclareSymbolFont{mcy}{U}{wncy}{m}{n}
\DeclareMathSymbol{\Sh}{\mathord}{mcy}{"58}

\DeclareFontFamily{U}{wncy}{}
\DeclareFontShape{U}{wncy}{m}{n}{<->wncyr10}{}
\DeclareSymbolFont{mcy}{U}{wncy}{m}{n}
\DeclareMathSymbol{\Ch}{\mathord}{mcy}{"51}

\title{Fibrations associated to smooth quotients of abelian varieties}

\author{Gary Martinez-Nuñez}
\address{Departamento de Matemáticas, Facultad de Ciencias, Universidad de Chile}
\email{gary.martinez@ug.uchile.cl}

\begin{document}

\begin{abstract}
	Let $A$ be an abelian variety and $G$ a finite group of automorphisms of $A$ fixing the origin such that $A/G$ is smooth. The quotient $A/G$ can be seen as a fibration over an abelian variety whose fibers are isomorphic to a product of projective spaces. We classify how the fibers are glued in the case when the fibers are isomorphic to a projective space and we prove that, in general, the quotient $A/G$ is a fibered product of such fibrations.
	
	\vspace{1em}
	
	\mbox{ \textbf{ Keywords:} Abelian varieties, fibrations, groups action}
\end{abstract}

\maketitle

		\mbox{\hspace{1.6em} \textbf{ MSC codes (2020):} 14L30, 14K99.} 

%
%
%
%

\section{Introduction}

\indent In \cite{ALA} R. Auffarth and G. Lucchini Arteche study smooth quotients of abelian varieties by finite groups whose action fixes the origin. Let $A$ be an abelian variety and $G$ a finite group of automorphisms of $A$ fixing the origin such that $A/G$ is smooth. Let $A_{0}$ be the connected component of $A^{G}$ that contains 0 and $P_{G}$ be the complementary abelian subvariety with respect to a $G$-invariant polarization. The action of $G$ on $A$ induces an action of $G$ on $P_{G}$. There exists a fibration\footnote{The quotient $A/P_{G}$ is isomorphic to $A_{0}/\left(A_{0}\cap P_{G}\right)$, the base of the fibration in Proposition 2.9 \cite{ALA}} $A/G\to A/P_{G}$ and the fibers are isomorphic to $P_{G}/G$ and smooth (Proposition 2.9 \cite{ALA}).

If $A_{0}=0$, or equivalently $A=P_{G}$, the fibration $A/G\to A/P_{G} \cong 0$ has just one fiber and this corresponds to $A/G$. This case was studied in \cite{ALA} for dimension $\geq 3$, and the case of dimension 2 was studied in \cite{Queso} and can be summarized in the following two results:

\begin{theorem}[Theorem 1.1 \cite{ALA} - Theorem 1.1 \cite{Queso}]\label{maintheorempaper}
	Let $A$ be an abelian variety and let $G$ be a (non trivial) finite group of automorphisms of $A$ that fix the origin. Then the following conditions are equivalent:
		\begin{enumerate}
			\item $A/G$ is smooth and the analytic representation of $G$ is irreducible.
			\item $A/G\cong \P^{n}$.
			\item There exists an elliptic curve $E$ such that $A\cong E^{n}$ and $(A,G)$ satisfies exactly one of the following:
				\begin{enumerate}[a)]
					\item\label{theoremalpha} $G\cong C^{n}\rtimes S_{n}$ where $C$ is a non-trivial (cyclic) subgroup of automorphims of $E$ that fix the origin; here the action of $C^{n}$ is coordinatewise and $S_{n}$ permutes the coordinates.
					\item\label{theorembeta} $G\cong S_{n+1}$ and acts on 
						\[A\cong \{(x_{0},x_{1},\dots,x_{n})\in E^{n+1} \mid x_{0}+x_{1}+\cdots+x_{n}=0\},\]
by permutations.
					\item\label{theoremgamma} $E\cong \C/\Z[i]$ and $G$ is the order 16 subgroup of $\GL_{2}(\Z)$ generated by:
		\[\left\{\begin{pmatrix}-1 & i+1 \\ 0 & 1\end{pmatrix},\begin{pmatrix}-i & i-1 \\ 0 & i\end{pmatrix},\begin{pmatrix}-1 & 0 \\ i-1 & 1\end{pmatrix}\right\},\]
acting on $A\cong E^{2}$ in the obvious way.
				\end{enumerate}
		\end{enumerate}
\end{theorem} 
	
	Even if $A_{0}=0$, the analytic representation of $G$ on $A$ might not always be irreducible. However, in this case the abelian variety $A$ and the group $G$ can be conveniently decomposed as in the next result:

\begin{theorem}[Theorem 1.3 \cite{ALA}]\label{descomposicionP_G}

	Let $A$ be an abelian variety and let $G$ be a finite subgroup of $\Aut_{0}(A)$ such that $A/G$ is smooth. Assume that $A_{0}=0$. There exist $A_{1},\dots,A_{r}$ abelian subvarieties of $A$ and $G_{i}\in\Aut_{0}(A_{i})$ such that $G\cong G_{1}\times\cdots\times G_{r}$, $A\cong A_{1}\times\cdots\times A_{r}$ and the pairs $(A_{i},G_{i})$ satisfy the conditions of Theorem \ref{maintheorempaper}. Moreover,
				\[A/G\cong A_{1}/G_{1}\times\cdots A_{r}/G_{r}.\]
\end{theorem}

	Thus, if $A_{0}=0$, the quotient $A/G$ is a product of projective spaces. In the case when $A_{0}\not=0$ the fibers of the fibration $A/G\to A/P_{G}$ are isomorphic to $P_{G}/G$. The group $G$ acts naturally on $P_{G}$ and gives a faithful representation of $G$ on $T_{0}(P_{G})$, the analytic representation of $G$ on $P_{G}$. Since $\dim(P_{G}^{G})=0$, by Theorem \ref{maintheorempaper} and Theorem \ref{descomposicionP_G} the fibers $P_{G}/G$ are isomorphic to a product of projective spaces.
	
	In this article we study how the fibers are glued in the case when $A_{0}\not= 0$, and therefore complete the classification of smooth quotients of abelian varieties by finite groups that fix the origin given by Auffarth and Lucchini Arteche.
	
	Smooth quotients of abelian varieties by finite groups whose action fixes the origin have appeared in other contexts, for example in the study of elliptic algebras introduced by Feigin and Odesskii in \cite{OdFeEA}. Much of the representation theory of these algebras is controlled by its characteristic variety and A. Chirvasitu, R. Kanda and S. P. Smith prove in \cite{CKSchar} that the characteristic variety of an elliptic algebra is isomorphic to $E^{g}/S$, with $E^{g}$ identified with the points of $E^{g+1}$ whose coordinates add up to 0 and $S$ is a subgroup of the permutation group $S_{g+1}$ generated by simple reflections. In particular, it is isomorphic to a smooth quotient of an abelian variety by a finite group whose action fixes the origin. In \cite{CKS}, the authors study more in detail the structure of the quotient variety $E^{g}/S$, which corresponds to a particular case of our Theorem \ref{maintheorem} below.  
	
	\begin{remark} Having completed such a classification a natural question that arises would be \textit{what happens in the case when $G$ does not fix a point of $A$?}. This question has been studied by the authors of \cite{ALA} in \cite{ALAtori}.
	\end{remark}
	
\subsection*{Main results} 
	
	If $A_{0}$ is trivial, the fibration has a trivial base and the quotient $A/G$ is isomorphic to a product of projective spaces, because of Theorem \ref{maintheorempaper} and Theorem \ref{descomposicionP_G}. However, if $A_{0}\not=0$ the base is non-trivial and the fibers are isomorphic to a product of projective spaces. In the case when the analytic representation of $G$ on $P_{G}$ is irreducible, i.e. when the fibers are isomorphic to a single projective space, we get the first result of this work.

	\begin{theorem}\label{maintheorem}
		Let $A$ be an abelian variety and $G$ a finite subgroup of $\Aut_{0}(A)$ such that $A/G$ is smooth. Let $A_{0}$ be the connected component of $A^{G}$ that contains 0 and let $P_{G}$ be the complementary abelian subvariety with respect to a $G$-invariant polarization. If the analytic representation of $G$ on $P_{G}$ is irreducible, there exists a trivialization
		
		\[\xymatrix{ A_{0}\times \P^{n} \ar[rr]^{/\Delta} \ar[d] & & A/G \ar[d] \\ A_{0} \ar[rr]_{/\Delta} & &  A / P_{G},  }\]
of the fibration $A/G\to A/P_{G}$, with $\Delta\cong A_{0}\cap P_{G}$, where $\Delta$ acts on $A_{0}$ by translation and on the fibers as one of the following:

		\begin{enumerate}
			\item \label{caso1} $\Delta\cong\{0\}$ and therefore acts trivially.
			\item \label{caso2} $\Delta\cong \Z/2\Z$ and acts with generator $\delta:[z_{0}:\cdots:z_{n}]\mapsto[z_{0}:-z_{1}:\cdots:(-1)^{n}z_{n}]$.
			\item \label{caso3} $\Delta\cong \Z/3\Z$ and acts with generator $\delta:[z_{0}:\cdots:z_{n}]\mapsto[z_{0}:\zeta_{3}z_{1}:\cdots:\zeta_{3}^{n}z_{n}]$.
			\item \label{caso4} $\Delta\cong \left(\Z/2\Z\right)^{2}$ and there exists a set of generators that act as
						\begin{align*} 
							\delta:[z_{0}:\cdots:z_{n}] &\mapsto [z_{0}:\cdots:(-1)^{n}z_{n}], \\
							\delta':[z_{0}:\cdots:z_{n}] &\mapsto [z_{n}:z_{n-1}:\cdots:z_{1}:z_{0}].
						\end{align*}

			\item \label{caso5} $\Delta$ is isomorphic to a subgroup of $\left(\Z/(n+1)\Z\right)^{2}$ and there exists a set of generators that act as
				\begin{align*} 
					\delta:[z_{0}:\cdots:z_{n}] &\mapsto [z_{0}:\zeta_{n+1}^{a}z_{1}:\cdots:\zeta_{n+1}^{an}z_{n}], \\
					\delta':[z_{0}:\cdots:z_{n}] &\mapsto [z_{n+1-b}:\cdots:z_{n-1}:z_{n}:z_{0}:\cdots:z_{n-b}],
				\end{align*}
with $a$ and $b$ positive integers, $b\lvert (n+1)$ and $a\lvert b$, where $z_{-1}:=z_{n}$.
		\end{enumerate}
	\end{theorem}
	
	In the case when the analytic representation of $G$ on $P_{G}$ is not irreducible, we prove that the quotient $A/G$ is isomorphic to a fibered product of fibrations as in Theorem \ref{maintheorem}.
	\begin{theorem}\label{maintheorem2}
		Let $A$ be an abelian variety and $G$ a finite non trivial subgroup of $\Aut_{0}(A)$ such that $A/G$ is smooth. Let $A_{0}$ be the connected component of $A^{G}$ that contains 0 and let $P_{G}$ be the complementary abelian subvariety with respect to a $G$-invariant polarization. There exist pairs $(B_{1},G_{1}),\dots,(B_{r},G_{r})$ of abelian varieties $B_{i}$ and subgroups $G_{i}$ of $\Aut_{0}(B_{i})$ such that 
		\begin{itemize}
			\item $G\cong G_{1}\times \cdots \times G_{r}$.
			\item The pairs $(B_{i},G_{i})$ satisfy the conditions of Theorem \ref{maintheorem}. In particular, if we denote by $B_{i,0}$ the connected component of $B_{i}^{G_{i}}$ that contains 0 and $P_{G_{i}}$ the complementary abelian subvariety with respect to a $G_{i}$-invariant polarization, there exists a fibration $B_{i}/G_{i}\to B_{i}/P_{G_{i}}$. 
			\item For all $i\in\{1,\dots,r\}$ $B_{i}/P_{G_{i}}\cong A/P_{G}$. 
			\item There exists an isomorphism
				\[ A/G\cong B_{1}/G_{1}\underset{A/P_{G}}{\times}\cdots\underset{A/P_{G}}{\times} B_{r}/G_{r}.\]
		\end{itemize}

	\end{theorem}
	
\subsection*{Structure of the article}
	In Section \ref{reduccionalcasoirreducible} we talk about general aspects of this classification and we prove Theorem \ref{maintheorem2}.
	 
	In Section \ref{casoirreducible} we give a proof of Theorem \ref{maintheorem}. In this section, the pairs $(P_{G},G)$ are separated into pairs of type $\alpha$, $\beta$ and $\gamma$, corresponding to the three examples in Theorem \ref{maintheorempaper} (cf. Definition \ref{tiposfibrados}). In Section \ref{tipoalpha} we study the pairs of type $\alpha$ and obtain the main result corresponding to points \ref{caso2}, \ref{caso3} and \ref{caso4} of Theorem \ref{maintheorem}. In Section \ref{tipobeta} we study the pairs of type $\beta$ and prove the main result corresponding to point \ref{caso5} of Theorem \ref{maintheorem}. The pairs of type $\gamma$ are studied in Section \ref{tipogamma}, the main result proving that the fibers are glued in the same way as the pairs of type $\alpha$.
	

%
%
	
\subsection*{Acknowledgements}
	I would like to thank Dr. Robert Auffarth and Dr. Giancarlo Lucchini Arteche for many enriching conversations. This work was partially funded by a scholarship of the Departamento de Matemática of the Universidad de Chile.

%
%


\section{Reduction to the irreducible case}\label{reduccionalcasoirreducible}

\indent  Let $A$ be an abelian variety and $G$ be a finite subgroup of $\Aut_{0}(A)$ such that $A/G$ is smooth. Let $A_{0}$ be the connected component of $A^{G}$ that contains 0 and $P_{G}$ be the complementary abelian subvariety with respect to a $G$-invariant polarization. As stated above, there exists a fibration $A/G\to A/P_{G}$ whose fibers are smooth and isomorphic to $P_{G}/G$. Consider the isogeny 
					\begin{align*}
						A_{0}\times P_{G} &\to A; \\
							(a,b) &\mapsto a-b,
					\end{align*}
with kernel
	\[\Delta:=\{(a,a)\in A_{0}\times P_{G} \mid a\in A_{0}\cap P_{G}\}.\]
	
	The subgroup $\Delta$ acts by translations coordinatewise on $A_{0}\times P_{G}$. The group $G$ acts on $A_{0}\times P_{G}$ coordinatewise and this action commutes with the action of $\Delta$ because the action of $\Delta$ is given by translations by $G$-invariant elements. We therefore obtain a commutative diagram
		\begin{equation}\label{trivializacion}
			\xymatrix{ A_{0}\times P_{G} \ar[rr]^{/\Delta} \ar[d]_{/G} & & A \ar[d]^{/G} \\ A_{0}\times \left(P_{G}/G\right) \ar[rr]_{/\Delta} \ar[d] & & A/G \ar[d] \\ A_{0} \ar[rr]_{/\Delta} & & A/ P_{G}. }
		\end{equation}

	In the bottom square of the diagram $A_{0}\times\left(P_{G}/G\right)\to A_{0}$ is the projection onto the first coordinate, and therefore that part of the diagram is a trivialization of the fibration $A/G\to A/P_{G}$ given by Proposition 2.9 in \cite{ALA}. The action of $\Delta$ passes to $A_{0}\times (P_{G}/G)$ and acts on $A_{0}$ by translation. Thus, giving a full classification of these fibrations is reduced to studying the actions of $\Delta$ on $P_{G}/G$ which are induced by translation by elements of $A_{0}\cap P_{G}$ on $P_{G}$. In other words, we need to understand the commutative diagram
			\begin{equation}\label{fibraenfibra1}
				\xymatrix{ P_{G} \ar[r]^{t} \ar[d]_{G} & P_{G} \ar[d]^{G} \\ P_{G}/G \ar[r]_{t_{G}} & P_{G}/G ,}
			\end{equation}
where $t$ is any translation by a non zero element of $A_{0}\cap P_{G}$ and $t_{G}$ is the induced morphism on the quotient. Notice that $A_{0}\cap P_{G}$ acts faithfully on $P_{G}/G$, and particularly, if $A_{0}\cap P_{G}\not=\{0\}$, the fibration $A/G\to A/P_{G}$ is not trivial. Then, giving a classification of the fibration $A/G\to A/P_{G}$ is reduced to studying faithful representations of $\Delta$ in $\Aut(P_{G}/G)$, where $P_{G}/G$ is a product of projective spaces. In the case when the analytic representation of $G$ on $P_{G}$ is irreducible the problem is reduced to studying faithful representations of $\Delta$ in $\PGL_{n+1}(\C)$, which are easier to understand than the representations of $\Delta$ in $\Aut(\P^{n_{1}}\times\cdots\times \P^{n_{r}})$. In Section \ref{casoirreducible} we study the case when the analytic representation of $G$ on $P_{G}$ is irreducible. Now, let us prove that the general case can be reduced to the irreducible case. As $\dim (P_{G}^G)=0$ and $P_{G}/G$ is smooth, by Theorem \ref{descomposicionP_G}, the abelian subvariety $P_{G}$ and the group $G$ can be decomposed. Then, there exist abelian subvarieties $P_{1},\dots P_{r}\subset P_{G}$ and subgroups $G_{1},\dots,G_{r}\leq \Aut_{0}(A)$ such that $P_{G}\cong P_{1}\times\cdots\times P_{r}$, $G\cong G_{1}\times\cdots\times G_{r}$ and 
	\[P_{G}/G\cong P_{1}/G_{1}\times\cdots \times P_{r}/G_{r}.\]
	Furthermore, the analytic representation of $G_{i}$ on $P_{i}$ is irreducible and, by Theorem \ref{maintheorempaper}, $P_{i}/G_{i}\cong \P^{n_{i}}$. Thus, the morphism 
		\begin{align*}
			A_{0}\times P_{1}\times \cdots \times P_{r} &\to A; \\
				(a,b_{1},\dots,b_{r}) &\mapsto a-(b_{1}+\dots+b_{r}).
		\end{align*} 
is an isogeny of kernel $\Delta$. Let us fix an inclusion $\iota:\Delta\to A_{0}$ given by $(a,b_{1},\dots,b_{r})\mapsto a$ and denote by $\Delta_{0}$ the image $\iota(\Delta)$. This morphism is well defined and is an inclusion, because for all $a\in A_{0}\cap P_{G}$ there are unique $b_{i}\in P_{i}$ such that $a=b_{1}+\dots+b_{r}$. Moreover, we have $\Delta_{0}=A_{0}\cap P_{G}$.

	The fibration $A/G\to A/P_{G}$ is trivialized by $A_{0}\times \P^{n_{1}}\times\cdots\times \P^{n_{r}}\to A_{0}$ as in the following diagram
			\begin{equation*}
				\xymatrix{ A_{0}\times P_{1}\times\cdots\times P_{r} \ar[rr]^{/\Delta} \ar[d]_{/G} & & A \ar[d]^{/G}  \\ 
					A_{0}\times \P^{n_{1}}\times\cdots\times \P^{n_{r}} \ar[d] \ar[rr]^{/\Delta} & & A/G \ar[d]  \\
					A_{0} \ar[rr]^{/\Delta_{0}} & & A/ P_{G}.}
			\end{equation*}
		
	Our purpose is understand to the morphism $A_{0}\times \P^{n_{1}}\times\cdots\times \P^{n_{r}} \to A/G$ in order to prove Theorem \ref{maintheorem2}:

	\begin{proof}[Proof of Theorem \ref{maintheorem2}]
	 Let $\mathcal{P}_{i}:=\prod_{k\in\{1,\dots,r\}-\{i\}}P_{k}$ be the product of the $P_{k}$ with $k\not=i$, which can be considered as an abelian subvariety of $P_{G}$, hence of $A$. Denote $B_{i}:=A/\mathcal{P}_{i}$. The abelian variety $\mathcal{P}_{i}$ acts on $A_{0}\times P_{1}\times\cdots\times P_{r}$ by translation on each coordinate. This action commutes with the action of $\Delta$. Hence, we have the commutative diagram
			\[\xymatrix{A_{0}\times P_{1}\times\cdots\times P_{r} \ar[rr]^-{/\Delta} \ar[d]_{/\mathcal{P}_{i}} & & A \ar[d]^{/\mathcal{P}_{i}} \\
					A_{0}\times P_{i} \ar[rr]^{} & & B_{i} .}\]

	Denote by $\Delta_{i}$ the kernel of the morphism $A_{0}\times P_{i} \to B_{i}$. As $\Delta\cap \mathcal{P}_{i}=0$, we have $\Delta_{i} \cong \Delta$. Therefore, the morphism $A_{0}\times P_{i}\to B_{i}$ is an isogeny. Moreover, the restriction to $P_{i}$ is an injection, because $P_{i}\cap \mathcal{P}_{i}=0$. The group $G_{i}$ acts on $B_{i}$ and this action is faithful. Denote by $B_{i,0}$ the connected component of $B_{i}^{G_{i}}$ that contains 0 and $P_{G_{i}}$ the complementary abelian subvariety by a $G_{i}$-invariant polarization. By Proposition 2.9 of \cite{ALA} there exists a fibration $B_{i}/G_{i}\to B_{i}/P_{G_{i}}$. On the one hand, the image of $A_{0}\times\{0\}$ in $B_{i}$ is contained in $B_{i,0}$, and in fact the image is equal to $B_{i,0}$. This gives a surjective morphism $A_{0}\to B_{i,0}/(B_{i,0}\cap P_{G_{i}})\cong B_{i}/P_{G_{i}}$. On the other hand, the image of $\{0\}\times P_{i}$ is contained in $P_{G_{i}}$. Since $P_{i}$ is injected into $P_{G_{i}}$, both are connected and have the same dimension, they are isomorphic. Hence, $P_{G_{i}}/G_{i}\cong P_{i}/G_{i}\cong \P^{n_{i}}$ is smooth and, by Proposition 2.9 in \cite{ALA}, the quotient $B_{i}/G_{i}$ is smooth. Then, we have the commutative diagram
					\begin{equation}\label{trivetale}
						\xymatrix{ A_{0}\times P_{i} \ar[r]^{/\Delta_{i}} \ar[d]_{/G_{i}} & B_{i} \ar[d]^{/G_{i}}  \\ 
					A_{0}\times \P^{n_{i}} \ar[d] \ar[r]^{/\Delta_{i}} & B_{i}/G_{i} \ar[d]  \\
					A_{0} \ar[r]^{} & B_{i}/P_{G_{i}},}
					\end{equation}
where all the morphisms are surjective. We claim that the kernel of the morphism $A_{0}\to B_{i}/P_{G_{i}}$ is $\Delta_{0}$. On the one hand, let $f:A_{0}\times P_{i}\to B_{i}$ be the morphism given by diagram \eqref{trivetale} and $a\in A_{0}$  such that $f(a,0)=\overline{0}\in B_{i,0}/(B_{i,0}\cap P_{i})$. Then, $f(a,0)\in B_{i,0}\cap P_{i}$ and there exist $(0,-b_{i})\in \{0\}\times P_{i}$ such that $f(0,-b_{i})=f(a,0)$, so $f(a,b_{i})=0$. This implies that $(a,b_{i})\in \Delta_{i}$, so $a\in \Delta_{0}$. On the other hand, if $a\in \Delta_{0}$ there exits $(b_{1},b_{2},\dots,b_{r})\in \prod_{k=1}^{n} P_{k}$ such that $a=b_{1}+b_{2}+\cdots+b_{r}$. In particular, $a-b_{i}\in \mathcal{P}_{i}$, hence $f(a,b_{i})=0$ in $B_{i}$ and $a\in \ker(f)$. Then, the kernel of the morphism $A_{0}\to B_{i}/P_{G_{i}}$ is $\Delta_{0}$ and \[B_{i,0}/(B_{i,0}\cap P_{i})\cong A_{0}/\Delta_{0},\] hence we can see $B_{i}/G_{i}$ as a fibration over $A_{0}/\Delta_{0}$ and therefore, since $\Delta_{0}=A_{0}\cap P_{G}$,  as a fibration over $A/P_{G}$ and, in particular, \[ B_{i}/P_{G_{i}}\cong A/P_{G}. \] \\

	So far, we have proved the three first statements of Theorem \ref{maintheorem2}. We need to prove the existence of the isomorphism \[ A/G\cong B_{1}/G_{1}\underset{A/P_{G}}{\times}\cdots\underset{A/P_{G}}{\times} B_{r}/G_{r}.\] The morphisms $A\to B_{i}$ induce morphisms $A/G\to B_{i}/G$. By Theorem \ref{descomposicionP_G}, the group $G_{j}$ acts trivially over $B_{i}$ when $j\not=1$, therefore $B_{i}/G=B_{i}/G_{i}$. The quotient $B_{i}/G_{i}$ is smooth because its fibers are isomorphic to $\P^{n_{i}}$, then the pairs $(B_{i},G_{i})$ satisfy the conditions of Theorem \ref{maintheorem}. Therefore, we have the commutative diagram given by the morphisms $A/G\to B_{i}/G_{i}$ and the vertical isomorphisms given by the diagram \eqref{trivetale} 

				\[\xymatrix@=1pt{
					     (A_{0}\times \P^{n_{1}}\times\cdots \times\P^{n_{r}})/\Delta \ar[rrr] \ar[rdd] \ar[ddd]_{\cong} \ar[drr] & & & (A_{0}\times \P^{n_{r}})/\Delta_{r} \ar[rdd] \ar@{->}'[dd]_{\cong}[ddd] & \\
					     & & \ar@{}[ru]|*=[@]{\hdots} (A_{0}\times\P^{n_{i}})/\Delta_{i} \ar@{}[dl]|*=[@]{\hdots} \ar[drr] \ar@{->}'[d]_{}[ddd] & \\
					     & (A_{0}\times \P^{n_{1}})/\Delta_{1} \ar[rrr] \ar[ddd]_{} & & & A_{0}/\Delta_{0} \ar[ddd]^{\cong} \\
					     A/G \ar@{->}'[rr][rrr] \ar[rdd] \ar[drr] & & & B_{r}/G_{r} \ar[rdd] & \\
					     & & \ar@{}[ru]|*=[@]{\hdots} B_{i}/G_{i} \ar@{}[dl]|*=0[@]{\hdots} \ar[drr] & \\
					     & B_{1}/G_{1} \ar[rrr] & & & A/P_{G} .}\]
Thus, giving a proof of the fact that $A/G$ is the fibered product of the $B_{i}/G_{i}$ over $A/P_{G}$ is equivalent to proving that $(A_{0}\times \P^{n_{1}}\times\cdots \times\P^{n_{r}})/\Delta$ is the fibered product of the $(A_{0}\times \P^{n_{i}})/\Delta_{i}$ over $A_{0}/\Delta_{0}$. Let us prove the second assertion. Let $Z$ be a variety with morphisms $q_{i}:Z\to (A_{0}\times \P^{n_{i}})/\Delta_{i}$ such that the following diagram commutes for all $i,j\in\{1,\dots,r\}$:

				\begin{equation}\label{prodfib}
					\xymatrix{ Z \ar@/_1pc/[ddr]_{q_{i}}  \ar@/^/[drr]^{q_{j}} \ar@{.>}[dr]|-{\Psi}& \\  & (A_{0}\times\P^{n_{1}}\times\cdots\times\P^{n_{r}})/\Delta_{} \ar[d]^-{p_{i}} \ar[r]_-{p_{j}} & \ar[d]_-{\varphi_{j}} (A_{0}\times\P^{n_{j}})/\Delta_{j} \\ & (A_{0}\times\P^{n_{i}})/\Delta_{i} \ar[r]^-{\varphi_{i}} & A_{0}/\Delta_{0}. }
				\end{equation}

	We need to prove the existence and the uniqueness of $\Psi$. Let us prove first the existence. Let $z\in Z$ and $q_{i}(z)=[(x_{i},\alpha_{i})]_{\Delta_{i}}$ with $\alpha_{i}\in\P^{n_{i}}$, where $[(x_{i},\alpha_{i})]_{\Delta_{i}}$ is the class of $(x_{i},\alpha_{i})$ modulo $\Delta_{i}$. Since for all $i,j\in\{1,\dots,r\}$
				\[\varphi_{i}q_{i}(z)=\varphi_{j}q_{j}(z),\]
we have for all $i\in\{2,\dots,r\}$
		\[\varphi_{1}q_{1}(z)=\varphi_{1}\left( [(x_{1},\alpha_{1})]_{\Delta_{1}} \right)=[x_{1}]_{\Delta_{0}}=[x_{i}]_{\Delta_{0}}=\varphi_{i}\left( [(x_{i},\alpha_{i})]_{\Delta_{i}} \right)=\varphi_{i}q_{i}(z).\]
	This implies that there exists $a_{i}\in\Delta_{0}$ such that $a_{i}*x_{1}=x_{i}$. Since $a_{i}\in \Delta_{0}=A_{0}\cap P_{G}$, there exists a unique $b_{ii}\in P_{i}$ such that $(a_{i},b_{ii})\in \Delta_{i}$. Then, there exists $\alpha'_{i}\in\P^{n_{i}}$ such that $(a_{i},b_{i,i})*[\left(x_{1},\alpha'_{i}\right)]_{\Delta_{i}}=[\left(x_{i},\alpha_{i}\right)]_{\Delta_{i}}$. Then, we have for all $i\in\{2,\dots,r\}$
				\[q_{i}(z)=[(x_{1},\alpha'_{i})]_{\Delta_{i}}.\] This allows us to define the morphism $\Psi:Z\to (A_{0}\times\P^{n_{1}}\times\cdots\times\P^{n_{r}})/\Delta$ given by
						\[ z\mapsto [(x_{1},\alpha'_{1},\dots,\alpha'_{r})]_{\Delta}, \]
where $\alpha_{1}=\alpha'_{1}$, which makes diagram \eqref{prodfib} commute for all $i,j\in\{1,\dots,r\}$. Let us prove the uniqueness. Let $\overline{\Psi}:Z\to(A_{0}\times\P^{n_{1}}\times\dots\times\P^{n_{r}})/\Delta$ be another morphism that makes diagram \eqref{prodfib} commute for all $i,j\in\{1,\dots,r\}$. Let $z\in Z$ and $\overline{\Psi}(z)=[(x,\beta_{1},\dots,\beta_{r})]_{\Delta}$ for any $(x,\beta_{1},\dots,\beta_{r})\in A_{0}\times \P^{n_{1}}\times\cdots\times\P^{n_{r}}$. Then,
						\[p_{i}\circ\overline{\Psi}(z)=p_{i}([(x,\beta_{1},\dots,\beta_{r})]_{\Delta})=[(x,\beta_{i})]_{\Delta_{i}}=[(x_{1},\alpha'_{i})]_{\Delta_{i}}=q_{i}(z),\]
because $p_{i}\circ\overline{\Psi}=q_{i}$. This implies that there exists $(a_{i},b_{i})\in \Delta_{i}$ such that $(a_{i},b_{i})*(x_{1},\alpha'_{i})=(x,\beta_{i})$. Since the action of $\Delta_{i}$ on the first coordianate is by translations, we have for all $i\in\{1,\dots,r\}$ that $x=x_{1}+a_{i}$ and hence there exists $a\in A_{0}\cap P_{G}$ such that $a=a_{i}$ for all $i\in\{1,\dots,r\}$. Then, as the decomposition of $a$ in $P_{G}$ is unique, we have $a=b_{1}+\cdots+b_{r}$. Then, we have 
						\[(a,b_{1},\dots,b_{r})*(x_{1},\alpha'_{1},\dots,\alpha'_{r})=(x,\beta_{1},\dots,\beta_{r}),\]
which implies that $[(x_{1},\alpha'_{1},\dots,\alpha'_{r})]_{\Delta}=[(x,\beta_{1},\dots,\beta_{r})]_{\Delta}$, and hence the uniqueness of $\Psi$ because
				\[\Psi(z)=[(x_{1},\alpha'_{1},\dots,\alpha'_{r})]_{\Delta}=[(x,\beta_{1},\dots,\beta_{r})]_{\Delta}=\overline{\Psi}(z).\]
Therefore, there exists a unique morphism $\Psi:Z\to A/G$ such that for all  $i,j\in\{1,\dots,r\}$ the following diagram commutes
				\[\xymatrix{ Z \ar@/_/[ddr]_{q_{i}}  \ar@/^/[drr]^{q_{j}} \ar[dr]|-{\Psi}& \\  & A/G \ar[d]^-{p_{i}} \ar[r]_-{p_{j}} & \ar[d]_-{\varphi_{j}} B_{j}/G_{j} \\ & B_{i}/G_{i} \ar[r]^-{\varphi_{i}} & A/P_{G}. }\]					 
	This proves that
				\[A/G\cong B_{1}/G_{1}\underset{A/P_{G}}{\times}\cdots\underset{A/P_{G}}{\times} B_{r}/G_{r}.\] Finally, we have proved Theorem \ref{maintheorem2}.

	\end{proof}

\begin{remark}
	The trivialization $A_{0}\times \P^{n_{i}}\to A_{0}$ of $B_{i}/G_{i}\to A/P_{G}$ given by diagram \eqref{trivetale} does not correspond to the one given in Theorem \ref{maintheorem}. That trivialization corresponds to the right square of the commutative diagram.
		\[ \xymatrix{A_{0}\times\P^{n_{i}} \ar[r] \ar[d] & B_{i,0}\times \P^{n_{i}} \ar[r] \ar[d] & B_{i}/G_{i} \ar[d] \\
					A_{0} \ar[r]^{} & B_{i,0} \ar[r] & B_{i,0}/(B_{i,0}\cap P_{G_{i}})\cong A/P_{G}.}\]
\end{remark}


\section{Irreducible case}\label{casoirreducible}

\indent As in the last section, $A$ will be an abelian variety, $G$ a finite subgroup of $\Aut_{0}(A)$ such that $A/G$ is smooth, $A_{0}$ the connected component of $A^{G}$ that contains $0$ and $P_{G}$ the complementary abelian subvariety of $A_{0}$ with respect to a $G$-invariant polarization. In this section we will assume that the analytic representation of $G$ in $P_{G}$ is irreducible. In this case, by Theorem \ref{maintheorempaper}, we know that $P_{G}/G\cong \P^{n}$. Diagram \eqref{trivializacion}, then becomes
	
	\[\xymatrix{ A_{0}\times P_{G} \ar[rr]^{/\Delta} \ar[d]_{/G} & & A \ar[d]^{/G} \\ A_{0}\times \P^{n} \ar[rr]_{/\Delta} \ar[d] & & A/G \ar[d] \\ A_{0} \ar[rr]_{/\Delta} & & A/P_{G}. }\]
And the study of the relation between the fibers is reduced to diagram \eqref{fibraenfibra1}, which becomes
	\begin{equation}\label{P_{G}afibra}
		\xymatrix{ P_{G} \ar[r]^{t} \ar[d] & P_{G} \ar[d] \\ \P^{n} \ar[r]_{t_{G}} & \P^{n}, }
	\end{equation}
where $t$ is a translation by an element of $A_{0}\cap P_{G}$ and $t_{G}$ is the induced function on the fibers that makes the diagram commute. The subvariety $P_{G}$ and the group $G$, by Theorem \ref{maintheorempaper}, are classified as pairs $(P_{G},G)$.
	\begin{definition}\label{tiposfibrados}
	Let $P_{G}$ and $G$ be as in Theorem \ref{maintheorempaper}. The pairs $(P_{G},G)$ satisfying \ref{theoremalpha}) will be called of \textit{type $\alpha$}, those satisfaying \ref{theorembeta}) will be called of \textit{type $\beta$} and those satisfying \ref{theoremgamma}) will be called of \textit{type $\gamma$}.
	\end{definition}

	In what follows we treat the three cases separately.


\subsection{Classification of the fibers with $(P_{G},G)$ of type $\alpha$}\label{tipoalpha}

\indent If $(P_{G},G)$ is a pair of type $\alpha$, the abelian variety $P_{G}$ is isomorphic to the self-product of an elliptic curve $E$ and $G$ is isomorphic to $C^{n}\rtimes S_{n}$, with $n=\dim(P_{G})$ and $C$ a non-trivial subgroup of $\Aut(E)$, and therefore a cyclic group of order $2,3,4$ or 6. The action of $C^{n}$ is coordinatewise and the action of $S_{n}$ permutes them.

The main result of this section is the following.
	\begin{proposition}\label{clasificaciontipoalpha}
		Let $A$ be an abelian variety and $G$ be a finite subgroup of $\Aut_{0}(A)$ such that $A/G$ is smooth. Let $A_{0}$ be the connected component of $A^{G}$ that contains 0 and $P_{G}$ be the complementary abelian subvariety with respect to a $G$-invariant polarization. If the analytic representation of $G$ on $P_{G}$ is irreducible and $(P_{G},G)$ is a pair of type $\alpha$, there exists a trivialization
		
		\[\xymatrix{ A_{0}\times \P^{n} \ar[rr]^{/\Delta} \ar[d] & & A/G \ar[d] \\ A_{0} \ar[rr]_{/\Delta} & & A/P_{G},  }\]
of the fibration $A/G\to A/P_{G}$, with $\Delta\cong A_{0}\cap P_{G}$, where $\Delta$ acts on $A_{0}$ by translations and on the fibers as one of the following:

		\begin{enumerate}
			\item \label{alpha1} $\Delta\cong\{0\}$ and therefore acts trivially.
			\item \label{alpha2} $\Delta\cong \Z/2\Z$ and acts with generator $\delta:[z_{0}:\cdots:z_{n}]\mapsto[z_{0}:-z_{1}:\cdots:(-1)^{n}z_{n}]$.
			\item \label{alpha3} $\Delta\cong \Z/3\Z$ and acts with generator $\delta:[z_{0}:\cdots:z_{n}]\mapsto[z_{0}:\zeta_{3}z_{1}:\cdots:\zeta_{3}^{n}z_{n}]$.
			\item \label{alpha4} $\Delta\cong \left(\Z/2\Z\right)^{2}$ and there exists a set of generators that act as
						\begin{align*} 
							\delta:[z_{0}:\cdots:z_{n}] &\mapsto [z_{0}:\cdots:(-1)^{n}z_{n}], \\
							\delta':[z_{0}:\cdots:z_{n}] &\mapsto [z_{n}:z_{n-1}:\cdots:z_{1}:z_{0}].
						\end{align*}
						
		\end{enumerate}
	\end{proposition}
	
	\begin{proof}

	The fixed points of $P_{G}$ by the action of $G$ are
				\[P_{G}^{G}\cong\{(x,\dots,x)\in E^{n} \mid x\in E^{C}\}\cong E^{C}.\]
	
	Let $\Delta$ be a non trivial subgroup of $P_{G}^{G}$. Recall that we want to study the induced morphism $t_{G}$ between the fibers as in diagram \eqref{P_{G}afibra}
		\[\xymatrix{ P_{G} \ar[r]^{t} \ar[d] & P_{G} \ar[d] \\ \P^{n} \ar[r]_{t_{G}} & \P^{n} }\]
where $t$ is a translation by an element of $\Delta$. The morphism $P_{G}\to\P^{n}$ can be decomposed as 
	\[\xymatrix{P_{G}\ar[r]^{/C^{n}} & (\P^{1})^{n}\ar[r]^{/S_{n}} & \P^{n}},\] 
where the function $(\P^{1})^{n}\mapsto\P^{n}$ is given by 
		\begin{align*}
				g:(\P^{1})^{n} &\mapsto \P^{n}; \\
			\left([x_{1}:y_{1}],\dots,[x_{n}:y_{n}]\right) &\to [Y^{n-i}p_{i}(x_{1}Y_{1},\dots,x_{n}Y_{n})]_{i}, 
		\end{align*}
with $Y=y_{1}\cdots y_{n}$, $Y_{i}=Y/y_{i}$, $p_{0}\equiv 1$ and 
		\[ p_{i}(z_{1},\dots,z_{n})=\sum_{k_{1}<\cdots<k_{i}}z_{k_{1}}\cdots z_{k_{i}}. \]

	Since the action of $\Delta$ on $\P^{n}$ is faithful, the action induced on $(\P^{1})^{n}$ is too. This reduces the problem to studying the commutative diagram
		\[\xymatrix{ (\P^{1})^{n} \ar[r]^{\overline{t}} \ar[d]_{g} & (\P^{1})^{n} \ar[d]^{g} \\ \P^{n} \ar[r]_{t_{G}} & \P^{n}, }\] 
where $\overline{t}$ is the morphism induced by $t$ over $(\P^{1})^{n}$. Since $t$ translates by the same element in each coordinate, we can see $\overline{t}$ as an element of $\Aut(\P^{1})$ acting in a diagonal way on $(\P^{1})^{n}$. Since the action of $\Delta$ on $(\P^{1})^{n}$ is faithful and diagonal, the induced action of $\Delta$ on $\P^{1}$ is faithful. Hence, the group $\Delta$ is isomorphic to a subgroup of $\Aut(\P^{1})=\PGL_{2}(\C)$. By Proposition 4.1 in  \cite{beauville2009finite}, two isomorphic subgroups of $\PGL_{2}(\C)$ are conjugate. Therefore, it is enough to study one arbitrary representation of $\Delta$ in $\PGL_{2}(\C)$.

	When $\Delta$ has order 1, the fibration is the trivial one. Then, we assume $\Delta$ non trivial. 
	
	If the order of $C$ is $2$, then $P_{G}^{G}\cong E[2]$ and $E[2]\cong \left(\Z/2\Z\right)^{2}$. Therefore, $\Delta$ has order $2$ or $4$. When the order of $\Delta$ is $2$, we can consider
				\[\Delta\cong\left\langle\begin{pmatrix} -1 & 0 \\ 0 & 1\end{pmatrix}\right\rangle,\]
in $\PGL_{2}(\C)$. Let $\overline{t}$ be the non trivial translation of $\Delta$. We are looking for $t_{G}\in\Aut(\P^{n})$ such that $t_{G}g=g\overline{t}$. A direct computation gives $t_{G}:\P^{n}\to\P^{n}$ given by $[z_{0}:\cdots:z_{n}]\to[z_{0}:-z_{1}:\cdots:(-1)^{n}z_{n}]$. Then, if $\Delta$ has order 2 
					\[\Delta \cong \left\langle\delta:[z_{0}:z_{1}:\cdots:z_{n}]\to[z_{0}:-z_{1}:\cdots:(-1)^{n}z_{n}]\right\rangle.\]
This is point \ref{alpha2} of Proposition \ref{clasificaciontipoalpha}.

	If $\Delta$ has order 4, and hence is isomorphic to the Klein group, we can consider 
					\[\Delta\cong\left\langle\begin{pmatrix} -1 & 0 \\ 0 & 1\end{pmatrix},\begin{pmatrix} 0 & 1 \\ 1 & 0\end{pmatrix}\right\rangle.\]
	Let $\overline{t},\overline{u}:P_{G}\to P_{G}$ be two translations by elements of $\Delta$ such that $\langle t, u\rangle\cong \Delta$. Then, up to a base change, we can set
			\[ t \to \begin{pmatrix} -1 & 0 \\ 0 & 1\end{pmatrix} \quad\textrm{ y }\quad u\to\begin{pmatrix} 0 & 1 \\ 1 & 0\end{pmatrix}. \]

	By the work done above we know that $t_{G}:\P^{n}\to \P^{n}$ is given by 
	\[[z_{0}:\cdots:z_{n}]\to[z_{0}:-z_{1}:\cdots:(-1)^{n}z_{n}].\] 
	Then, we just need to determine $u_{G}:\P^{n}\to\P^{n}$ such that $u_{G}g=gu$. A direct computation gives $u_{G}:\P^{n}\to \P^{n}$ given by $[z_{0}:\cdots:z_{n}]\to[z_{n}:z_{n-1}:\cdots:z_{1}:z_{0}]$. Then, up to a change of basis, the action of $\Delta$ on $\P^{n}$ satisfies
			\[ \Delta\cong\left\langle\delta:[z_{0}:\cdots:z_{n}] \to [z_{0}:-z_{1}:\cdots:(-1)^{n}z_{n}]\, , \,\delta':[z_{0}:\cdots:z_{n}] \to [z_{n}:z_{n-1}:\cdots:z_{1}:z_{0}] \right\rangle. \]
This is point \ref{alpha4} of Proposition \ref{clasificaciontipoalpha}.
			
	When the order of $C$ is $3$ there is only one possible elliptic curve, the one with an action of $\Z/6\Z$. In this case $E^{C}\cong \Z/3\Z$, because there are only three fixed points by the action of $C$, hence $\Delta\cong \Z/3\Z$. Then, we can consider 
			\[\Delta\cong\left\langle\begin{pmatrix} \zeta_{3} & 0 \\ 0 & 1\end{pmatrix}\right\rangle.\]

	Let $t$ be a translation induced by any non trivial element of $\Delta$ and $\overline{t}:(\P^{1})^{n}\to(\P^{1})^{n}$ the morphism induced by $t$. Then, up to a base change, we can set 
			\[t_{G} \to \begin{pmatrix} \zeta_{3} & 0 \\ 0 & 1\end{pmatrix}.\]
			
	 We are looking for $t_{G}$ such that $t_{G}\circ g= g\circ\overline{t}$. A direct computation gives $t_{G}:\P^{n}\to\P^{n}$ given by $[z_{0}:\cdots:z_{n}]\to[z_{0}:\zeta_{3}z_{1}:\cdots:\zeta_{3}^{n}z_{n}]$. Then,
				\[\Delta\cong\left\langle \delta:[z_{0}:\cdots:z_{n}]\to[z_{0}:\zeta_{3}z_{1}:\cdots:\zeta_{3}^{n}z_{n}] \right\rangle.\]
This is point \ref{alpha3} of Proposition \ref{clasificaciontipoalpha}.
				
	If $C$ has order $6$, the elliptic curve $E$ is the same as in the last case. But in this case there is only one fixed point, hence $\Delta\cong \{\id\}$ and we are done.

	If the order of $C$ is $4$, there are only two points on $E$ fixed by the action of $G$, hence $\Delta\cong\Z/2\Z$. This case has already been studied in the case when $C$ has order $2$.

	All these results prove Proposition \ref{clasificaciontipoalpha}, which corresponds to points \ref{caso1}, \ref{caso2}, \ref{caso3} and \ref{caso4} of Theorem \ref{maintheorem}.

	\end{proof}


\subsection{Classification of the fibers with $(P_{G},G)$ of type $\beta$}\label{tipobeta}

	The pairs $(P_{G},G)$ of type $\beta$ are composed of the abelian variety $P_{G}$ isomorphic to
					\[\{(x_{0},\dots,x_{n})\in E^{n+1} \mid x_{0}+\cdots+x_{n}=0\},\]
and the group $G$ isomorphic to $S_{n+1}$. 
	
	In this case the next proposition holds. 

\begin{proposition}\label{clasificaciontipobeta}
	Let $A$ be an abelian variety and $G$ a finite subgroup of $\Aut_{0}(A)$ such that $A/G$ is smooth. Let $A_{0}$ be the connected component of $A^{G}$ that contains 0 and $P_{G}$ be the complementary abelian subvariety with respect to a $G$-invariant polarization. If the analytic representation of $G$ on $P_{G}$ is irreducible and $(P_{G},G)$ is a pair of type $\beta$, there exists a trivialization 
		
		\[\xymatrix{ A_{0}\times \P^{n} \ar[rr]^{/\Delta} \ar[d] & & A/G \ar[d] \\ A_{0} \ar[rr]_{/\Delta} & & A/P_{G}  ,}\]
of the fibration $A/G\to A/P_{G}$, with $\Delta\cong A_{0}\cap P_{G}$, where $\Delta$ acts on $A_{0}$ by translations and there exists a set of generators such that $\Delta$ acts on the fibers as
				\begin{align*} 
					\delta:[z_{0}:\cdots:z_{n}] &\mapsto [z_{0}:\zeta_{n+1}^{a}z_{1}:\cdots:\zeta_{n+1}^{an}z_{n}], \\
					\delta':[z_{0}:\cdots:z_{n}] &\mapsto [z_{n+1-b}:\cdots:z_{n-1}:z_{n}:z_{0}:\cdots:z_{n-b}],
				\end{align*}
with $a$ and $b$ positive integers, $b\lvert (n+1)$ and $a\lvert b$, where $z_{-1}:=z_{n}$.
\end{proposition}
	
	In \cite{auffarth2017note}, R. Auffarth proves that we can see the quotient $P_{G}/G$ as
			\[ \left\{ D \in \Div(E) \mid D\sim (n+1)[0], \, D\geq 0\right\}, \]
which is the complete linear system of $(n+1)[0]$. Then, the morphism $P_{G}\to \P^{n}\cong P_{G}/G$ is given by
			\begin{align*} 
P_{G} &\to \lvert (n+1)[0] \rvert; \\
				 (x_{0},\dots,x_{n}) &\mapsto [x_{0}]+\cdots+[x_{n}].
			\end{align*}

	On the other hand, the fixed points of $P_{G}$ by the action of $G$ are 
				\[P_{G}^{G}=\{(x,\dots,x)\in E^{n+1} \mid x\in E[n+1]\}\cong E[n+1].\]
	Let $(x,\dots,x),(y,\dots,y)\in P_{G}^{G}$ be two generators of $P_{G}^{G}$. Denote by $t$ and $u$ the translations by $(x,\dots,x)$ and $(y,\dots,y)$, respectively, in $P_{G}$. Let $t_{G}$ and $u_{G}$ be the induced morphism by $t$ and $u$ on the fibers. We can see these morphisms as automorphisms of $\P^{n}$, as in the diagram
	\[\xymatrix{ P_{G} \ar[rr]^{t,u} \ar[d]_{} & & P_{G} \ar[d]^{} \\ \lvert (n+1)[0]\rvert \ar[rr]_{t_{G},u_{G}} & &  \lvert (n+1)[0]\rvert
		      .}\]
\indent	The following results will tell us that there is only one representation of $\Delta=P_{G}^{G}$ in $\PGL_{n+1}(\C)$, up to a base change. 		      
	\begin{lemma}\label{Fix}
		Let $(x,\dots,x),(y,\dots,y)\in P_{G}^{G}$ be generators of $P_{G}^{G}$. Let $t$ and $u$ be the translations by $(x,\dots,x)$ and $(y,\dots,y)$, respectively. If $t_{G}$ and $u_{G}$ are the induced morphisms on the quotient $P_{G}/G$, there exist $x_{t},y_{u}\in E$ such that
				\[\Fix(t_{G})=\{[x_{t}+my]+[x_{t}+my+x]+\cdots+[x_{t}+my+nx] \mid m\in\{0,\dots,n\}\},\]
			and
				\[\Fix(u_{G})=\{[y_{u}+mx]+[y_{u}+mx+y]+\cdots+[y_{u}+mx+ny] \mid m\in\{0,\dots,n\}\}.\]
		Moreover, $\Fix(t_{G})\cap\Fix(u_{G})=\emptyset$.
	\end{lemma}
	
	\begin{proof}
		The image of any point of $P_{G}$ in $\lvert(n+1)[0]\rvert$ is fixed by $t_{G}$ if and only if $t(x_{0},\dots,x_{n})=(x_{\sigma(0)},\dots,x_{\sigma(n)})$ for some $\sigma\in S_{n+1}$. Suppose that $\sigma=\sigma_{1}\sigma_{2}$ with $\sigma_{1}$ and $\sigma_{2}$ disjoint cycles and $\sigma_{1}$ non trivial. Then, we have $x_{\sigma_{1}^{i}(k)}+x=x_{\sigma_{1}^{i+1}(k)}$ for all $i\in\{0,\dots,\textrm{ord}(\sigma_{1})-1\}$, where $k$ is the start of $\sigma_{1}$. This implies that $\textrm{ord}(\sigma_{1})x=0$ and therefore $\textrm{ord}(x)\lvert\textrm{ord}(\sigma_{1})$. Since the order of $x$ is $n+1$, we have $\sigma=\sigma_{1}$. In this case, $x_{\sigma^{i}(0)}=x_{0}+ix$ and $(x_{0},x_{1},\dots,x_{n})\in P_{G}^{G}$, then
				\[0=\sum_{i=0}^{n}x_{\sigma^{i}(0)}=\sum_{i=0}^{n}(x_{0}+ix)=(n+1)x_{0}+\frac{n(n+1)}{2}x.\]
This equation in the variable $x_{0}$ has a solution in the elliptic curve and the difference between two of them is an element of $(n+1)$-torsion of the elliptic curve, so there are exactly $(n+1)^{2}$ different solutions. Let $x_{t}$ be a solution of the last equation. The divisors 
	\[D_{m}:=[x_{t}+my]+[x_{t}+my+x]+\cdots+[x_{t}+my+nx],\] 
with $m\in\{0,\dots,n\}$, are all fixed by the action of $t_{G}$ and they are all different because if $D_{m}=D_{\tilde{m}}$, there would exist $k\in\{0,\dots,n\}$ such that \[ x_{t}+\tilde{m}y =x_{t}+my+kx \] and hence $kx+(m-\tilde{m})y=0$. Since $x$ and $y$ generate the $(n+1)$-torsion of $E$, it follows that $k=0$ and $m=\tilde{m}$. Let $x'\in E$ be another solution and $D:=[x']+[x'+x]+\cdots+[x'+nx]$ the divisor associated to this solution. Since the difference between two solutions is an element of $(n+1)$-torsion, there exist $m,k\in\{0,\dots,n\}$ such that $x'=x_{t}+my+kx$, hence
			\begin{align*} 
				D &= [x']+[x'+x]+\cdots+[x'+nx] \\
				  &= [x_{t}+my+kx]+[x_{t}+my+(k+1)x]+\cdots+[x_{t}+my+(k+n)x] \\
				  &= [x_{t}+my]+[x_{t}+my+x]+\cdots+[x_{t}+my+nx]=D_{m}.
			\end{align*} 
 Then, the set of fixed points by $t_{G}$ is finite and moreover
		\[\Fix(t_{G})=\{[x_{t}+my]+[x_{t}+my+x]+\cdots+[x_{t}+my+nx] \mid m\in\{0,\dots,n\}\}.\]
		The proof is analogous for $\textrm{Fix}(u_{G})$.
		
		Suppose that $\textrm{Fix}(t_{G})\cap\textrm{Fix}(u_{G})\not=\emptyset$. Then, there exists $m\in\{0,\dots,n\}$ such that $[x_{t}+my]+[x_{t}+my+x]+\cdots+[x_{t}+my+nx]\in\textrm{Fix}(u_{G})$, therefore $u(x_{t}+my,x_{t}+my+x,\dots,x_{t}+my+nx)$ is a permutation of $(x_{t}+my,x_{t}+my+x,\dots,x_{t}+my+nx)$, so there exists $k\in\{0,\dots,n\}$ such that $x_{t}+my=x_{t}+(m+1)y+kx$. This occurs only if $y+kx=0$, but $x$ and $y$ generate $P_{G}^{G}$. Then, $\Fix(t_{G})\cap\Fix(u_{G})=\emptyset$.
		
	\end{proof}
	
	By Lemma \ref{Fix} we have that $t_{G}$ and $u_{G}$ have $n+1$ fixed points and these sets are disjoint. The following proposition allows us to conclude that the faithful representations of $\Delta=P_{G}^{G}$ in $\PGL_{n+1}(\C)$ induced by the action of $\Delta$ on $P_{G}$ are unique up to conjugation.
	
	In what follows, for any two elements $M$ and $N$ of $\GL_{n}(\C)$, $M\equiv N$ denotes that they have the same image in $\PGL_{n}(\C)$ by the projection.  
	
	\begin{proposition}\label{zsobrenzcuadrado}
		Let $\phi,\psi\in\PGL_{n}(\C)$ be two classes such that
		\begin{itemize} 
			\item both of them have order $n$ and commute,
			\item $\lvert\Fix(\phi)\rvert=\lvert\Fix(\psi)\rvert=n$,
			\item  $\Fix(\phi)\cap\Fix(\psi)=\emptyset$. 
		\end{itemize}
		There exists $\xi\in\PGL_{n}(\C)$ such that
			\[ \xi\phi\xi^{-1}\equiv\begin{pmatrix} 1 & 0 & 0 & \cdots & 0 \\ 0 & \zeta_{n} & 0 & \cdots & 0 \\ 0 & 0 & \zeta_{n}^{2} & \cdots & 0  \\ 0 & 0 & 0 & \ddots & 0 \\ 0 & 0 & 0 & \cdots & \zeta_{n}^{n-1} \end{pmatrix} \quad\quad\textrm{and}\quad\quad \xi\psi\xi^{-1}\equiv\begin{pmatrix} 0 & 1 & 0 & \cdots & 0 \\ 0 & 0 & 1 & \cdots & 0 \\ 0 & 0 & 0 & \ddots & 0  \\ 0 & 0 & 0 & \cdots & 1 \\ 1 & 0 & 0 & \cdots & 0 \end{pmatrix}.\]
	\end{proposition}

First we will prove the following lemma.

\begin{lemma}\label{conjugaciondiagonal}
	Let
		\[M\equiv \begin{pmatrix} 0 & v_{1} & 0 & \cdots & 0 \\ 0 & 0 & v_{2} & \cdots & 0 \\ 0 & 0 & 0 & \ddots & 0  \\ 0 & 0 & 0 & \cdots & v_{n-1} \\ v_{0} & 0 & 0 & \cdots & 0 \end{pmatrix} \quad \textrm{ and } \quad N\equiv \begin{pmatrix} 0 & 1 & 0 & \cdots & 0 \\ 0 & 0 & 1 & \cdots & 0 \\ 0 & 0 & 0 & \ddots & 0  \\ 0 & 0 & 0 & \cdots & 1 \\ 1 & 0 & 0 & \cdots & 0 \end{pmatrix},\]
be two classes in $\PGL_{n}(\C)$ with $v_{0}\cdots v_{n-1}=1$. There exists a diagonal matrix $D\in\GL_{n}(\C)$ such that $D^{-1}MD\equiv N$.
\end{lemma}

\begin{proof}  The matrix given by

			\[D=\begin{pmatrix} v_{0} & 0 & \cdots & 0 & 0 \\ 0 & v_{0}v_{1} & \cdots & 0 & 0 \\ 0 & 0 & \cdots & 0 & 0  \\ 0 & 0 & \ddots & v_{0}v_{1}\cdots v_{n-2} & 0 \\ 0 & 0 & \cdots & 0 & 1 \end{pmatrix},\]
works.
\end{proof}

	\begin{proof}[Proof of Proposition \ref{zsobrenzcuadrado}]
		Let $p_{1},\dots,p_{n}\in\Fix(\phi)$ and let $v_{1},\dots,v_{n}\in\C^{n}$ be a lifting of $p_{i}$ for $i\in\{1,\dots,n\}$. Let $\Phi$ be a lifting of $\phi$ to $\GL_{n}(\C)$. Since $\phi$ has exactly $n$ fixed points, we have that $\Phi$ has $n$ different eigenvalues and $\{v_{1},\dots,v_{n}\}$ is a base of $\C^{n}$. Then, up to a base change, 
			\[\Phi\equiv\begin{pmatrix} \phi_{1} & 0 & 0 & \cdots & 0 \\ 0 & \phi_{2} & 0 & \cdots & 0 \\ 0 & 0 & \phi_{3} & \cdots & 0  \\ 0 & 0 & 0 & \ddots & 0 \\ 0 & 0 & 0 & \cdots & \phi_{n} \end{pmatrix},\]
		where the $\phi_{i}$ are the eigenvalues of $\Phi$.
		
		Let us prove that $\psi$ acts on $\textrm{Fix}(\phi)$. Let $p\in\textrm{Fix}(\phi)$. Since $\psi$ and $\phi$ commute, we have 
			\[\phi\psi(p)=\psi\phi(p)=\psi(p).\]
This implies that $\psi$ acts on $\Fix(\phi)$. This allows us to see $\psi$ as an element of the permutation group of $\Fix(\phi)$. We claim that $\psi$ is an $n$-cycle. Otherwise, we can write $\psi$ as a product of disjoint cyles and let $\sigma:=(p_{a_{1}},\dots,p_{a_{m}})$ and $\tau:=(p_{b_{1}},\dots,p_{b_{l}})$ be two of these cycles. Let $\Psi\in\GL_{n}(\C)$ be a lifting of $\psi$ written in the basis $\{v_{1},\dots,v_{n}\}$. For some $\psi_{a_{1}},\dots,\psi_{a_{m}},\psi_{b_{1}},\dots,\psi_{b_{l}}\in\C^{*}$, we have
		\[\Psi(v_{a_{i}})=\begin{cases} \psi_{a_{i}}v_{a_{i+1}}\textrm{ if } 1\leq i\leq m-1, \\ \psi_{a_{m}}v_{a_{1}} \quad \textrm{ if } i=m, \end{cases}\]
	and
		\[\Psi(v_{b_{i}})=\begin{cases} \psi_{b_{i}}v_{b_{i+1}}\textrm{ if } 1\leq i\leq l-1, \\ \psi_{b_{l}}v_{b_{1}} \quad \textrm{ if } i=l. \end{cases}\]
	
		Let $\Sigma\in\GL_{n}(\C)$ be a lifting of $\sigma$ and $T\in\GL_{n}(\C)$ be a lifting of $\tau$. Since $\sigma$ and $\tau$ are in $\PGL_{n}(\C)$ we can assume that $\det(\Sigma)=\det(T)=1$, hence $\psi_{a_{1}}\cdots\psi_{a_{m}}=1$ and $\psi_{b_{1}}\cdots\psi_{b_{l}}=1$. Now, define 
				\[v_{a}:=v_{a_{1}}+\psi_{a_{1}}v_{a_{2}}+\psi_{a_{1}}\psi_{a_{2}}v_{a_{3}}+\cdots+\psi_{a_{1}}\cdots\psi_{a_{m-1}}v_{a_{m}},\]
and
				\[v_{b}:=v_{b_{1}}+\psi_{b_{1}}v_{b_{2}}+\psi_{b_{1}}\psi_{b_{2}}v_{b_{3}}+\cdots+\psi_{b_{1}}\cdots\psi_{b_{l-1}}v_{b_{l}}.\]
		These vectors are linearly independent, because they are in two different subspaces with trivial intersection and none is null. Moreover,
				\[ \Psi(v_{a}) = v_{a} \quad \mathrm{and}\quad \Psi(v_{b}) = v_{b}.\]	
Therefore, the projections of $v_{a}$ and $v_{b}$ in $\P^{n-1}$ are different, fixed by $\psi$ and have the same eigenvalue. This is a contradiction, because $\psi$ has $n$ fixed points and any lifting of $\psi$ has $n$ different eigenvalues, one for each fixed point. 

	Since $\psi$ is a $n$-cycle, up to reordering the basis $\{v_{1},\dots,v_{n}\}$, we have that 
				\[\Psi\equiv\begin{pmatrix} 0 & \psi_{2} & 0 & \cdots & 0 \\ 0 & 0 & \psi_{3} & \cdots & 0 \\ 0 & 0 & 0 & \ddots & 0  \\ 0 & 0 & 0 & \cdots & \psi_{n} \\ \psi_{1} & 0 & 0 & \cdots & 0 \end{pmatrix} \qquad\text{ and }\qquad\Phi\equiv\begin{pmatrix} \phi_{1} & 0 & 0 & \cdots & 0 \\ 0 & \phi_{2} & 0 & \cdots & 0 \\ 0 & 0 & \phi_{3} & \cdots & 0  \\ 0 & 0 & 0 & \ddots & 0 \\ 0 & 0 & 0 & \cdots & \phi_{n} \end{pmatrix}.\]

	By Lemma \ref{conjugaciondiagonal}, up to a base change by a diagonal matrix, we have that $\psi_{i}=1$ for all $i\in\{1,\dots,n\}$. Since $\phi\in\PGL_{n}(\C)$ we can put $\phi_{1}=1$ and by the commutativity of $\psi$ and $\phi$ whe have that			
		\[\begin{pmatrix} 0 & \phi_{2} & 0 & \cdots & 0 \\ 0 & 0 & \phi_{3} & \cdots & 0 \\ 0 & 0 & 0 & \ddots & 0  \\ 0 & 0 & 0 & \cdots & \phi_{n} \\ 1 & 0 & 0 & \cdots & 0 \end{pmatrix}\equiv\begin{pmatrix} 0 & 1 & 0 & \cdots & 0 \\ 0 & 0 & \phi_{1} & \cdots & 0 \\ 0 & 0 & 0 & \ddots & 0  \\ 0 & 0 & 0 & \cdots & \phi_{n-1} \\ \phi_{n} & 0 & 0 & \cdots & 0 \end{pmatrix},\]
		therefore there exists $\lambda\in\C^{*}$ such that $\phi_{2}=\lambda$, $1=\lambda \phi_{n}$ and $\phi_{i}=\lambda \phi_{i-1}$ for all $i\in\{3,\dots,n\}$. Then $\phi_{i}=\lambda^{i-1}$. Since all $\phi_{i}$ are different pairwise, we have that $\lambda$ is a primitive $n$-th root of unity. This proves the proposition.
	
	\end{proof}
	
	\begin{proof}[Proof of Proposition \ref{clasificaciontipobeta}]
	Recall that we want to classify all the trivializations of the form
		\[\xymatrix{ A_{0}\times \P^{n} \ar[rr]^{/\Delta} \ar[d] & & A/G \ar[d] \\ A_{0} \ar[rr]_{/\Delta} & & A/P_{G},  }\]
with $\Delta\cong A_{0}\cap P_{G}$, where $A_{0}\cap P_{G}\subset P_{G}^{G}$. It depends on how much the abelian varieties $A_{0}$ and $P_{G}$ intersect each other. We will assume $A_{0}\cap P_{G}=P_{G}^{G}$. This allows us to study the abstract group $\Delta$ as $\left\langle t_{G}, u_{G} \right\rangle$, because $\Delta$ acts faithfully on $P_{G}/G$, and take $\left\langle t_{G}, u_{G} \right\rangle\cong P_{G}^{G}$. The morphisms $t_{G}, u_{G}$ are automorphisms of the linear system $\lvert (n+1)[0]\rvert$ which can be identified with $\P^{n}$, so we can see them as commuting elements of $\PGL_{n+1}(\C)$. By Lemma \ref{Fix} we have that $\textrm{Fix}(t_{G})=n+1$, $\textrm{Fix}(u_{G})=n+1$ and $\textrm{Fix}(t_{G})\cap\textrm{Fix}(u_{G})=\emptyset$. Then, by Proposition \ref{zsobrenzcuadrado} we have that, up to a base change,

			\[ \phi_{t_{G}}\equiv\begin{pmatrix} 1 & 0 & 0 & \cdots & 0 \\ 0 & \zeta_{n+1} & 0 & \cdots & 0 \\ 0 & 0 & \zeta_{n+1}^{2} & \cdots & 0  \\ 0 & 0 & 0 & \ddots & 0 \\ 0 & 0 & 0 & \cdots & \zeta_{n+1}^{n} \end{pmatrix} \textrm{ and } \phi_{u_{G}}\equiv\begin{pmatrix} 0 & 1 & 0 & \cdots & 0 \\ 0 & 0 & 1 & \cdots & 0 \\ 0 & 0 & 0 & \ddots & 0  \\ 0 & 0 & 0 & \cdots & 1 \\ 1 & 0 & 0 & \cdots & 0 \end{pmatrix},\]
Hence $\Delta$, up to a base change, is generated by the automorphisms
			\begin{align*} 
					\delta:[z_{0}:\cdots:z_{n}] &\mapsto [z_{0}:\zeta_{n+1}z_{1}:\cdots:\zeta_{n+1}^{n}z_{n}], \\
					\delta':[z_{0}:\cdots:z_{n}] &\mapsto [z_{n}:z_{0}:z_{1}:\cdots:z_{n-1}].
			\end{align*}

	This representation does not depend on the set of generators of $\left\{ t_{G},u_{G}\right\}$, because if we take another pair of generators $t_{G}'$ and $u_{G}'$, everything done above remains the same. 
	
	All the subgroups of $\left(\Z/(n+1)\Z\right)^{2}$ have the form $\langle t_{G}^{a},u_{G}^{b} \rangle$ in some $\left(\Z/(n+1)\Z\right)$-basis and for some $a,b\in\Z$ with $b\lvert (n+1)$ and $a\lvert b$. Thus, if $\Delta$ is isomorphic to any proper subgroup of $P_{G}^{G}$ the action of $\Delta$ on $P_{G}/G$ is generated by
				\begin{align*} 
					\delta:[z_{0}:\cdots:z_{n}] &\mapsto [z_{0}:\zeta_{n+1}^{a}z_{1}:\cdots:\zeta_{n+1}^{an}z_{n}], \\
					\delta':[z_{0}:\cdots:z_{n}] &\mapsto [z_{n+1-b}:\cdots:z_{n-1}:z_{n}:z_{0}:\cdots:z_{n-b}].
				\end{align*}		
This proves Proposition \ref{clasificaciontipobeta}, which corresponds to the case \ref{caso5} of Theorem \ref{maintheorem}.
	\end{proof}


\subsection{Classification of the fibers with $(P_{G},G)$ of type $\gamma$}\label{tipogamma}

	There is just one pair $(P_{G},G)$, when the abelian variety $P_{G}$ is isomorphic to $E^{2}$, with $E$ the elliptic curve $\C/\Z[i]$, and $G$ is the group generated by the set 
		\[\left\{ \begin{pmatrix} -1 & 1+i \\ 0 & 1 \end{pmatrix},\begin{pmatrix} -i & i-1 \\ 0 & i \end{pmatrix},\begin{pmatrix} -1 & 0 \\ i-1 & 1 \end{pmatrix} \right\}.\]
	This case can be reduced to the one studied in Section \ref{tipoalpha} and we get the following result.
	\begin{proposition}\label{clasificaciontipogamma}
		Let $A$ be an abelian variety and $G$ a finite subgroup of $\Aut_{0}(A)$ such that $A/G$ is smooth. Let $A_{0}$ be the connected component of $A^{G}$ that contains 0 and $P_{G}$ be the complementary abelian subvariety with respect to a $G$-invariant polarization. If the analytic representation of $G$ on $P_{G}$ is irreducible and $(P_{G},G)$ is a pair of type $\gamma$, there exists a trivialization 
		
		\[\xymatrix{ A_{0}\times \P^{2} \ar[rr]^{/\Delta} \ar[d] & & A/G \ar[d] \\ A_{0} \ar[rr]_{/\Delta} & & A/P_{G},  }\]
of the fibration $A/G\to A/P_{G}$, with $\Delta\cong A_{0}\cap P_{G}$, where $\Delta$ acts on $A_{0}$ by translations and on the fibers as one of the following:

		\begin{enumerate}
			\item $\Delta\cong\{0\}$ and therefore acts trivially.
			\item $\Delta\cong \Z/2\Z$ and acts with generator $\delta:[z_{0}:z_{1}:z_{2}]\mapsto[z_{0}:-z_{1}:z_{2}]$.
			\item $\Delta\cong \Z/2\Z\times\Z /2\Z$ and there exists a set of generators that act as
						\begin{align*} 
							\delta:[z_{0}:z_{1}:z_{n}] &\mapsto [z_{0}:-z_{1}:z_{2}], \\
							\delta':[z_{0}:z_{1}:z_{n}] &\mapsto [z_{2}:z_{1}:z_{0}].
						\end{align*}
		\end{enumerate}
	\end{proposition}

	\begin{proof}

	The fixed points by the action of $G$ are the points of the form $(a+bi,c+di)\in E^{2}$ such that
		\begin{align*}
			\begin{pmatrix} -1 & 1+i \\ 0 & 1 \end{pmatrix} \begin{pmatrix} a+bi \\ c+di \end{pmatrix} &\equiv_{\Z[i]} \begin{pmatrix} a+bi \\ c+di \end{pmatrix}, \\ 
			\begin{pmatrix} -i & i-1 \\ 0 & i \end{pmatrix} \begin{pmatrix} a+bi \\ c+di \end{pmatrix} &\equiv_{\Z[i]} \begin{pmatrix} a+bi \\ c+di \end{pmatrix}, \\ 
			\begin{pmatrix} -1 & 0 \\ i-1 & 1 \end{pmatrix} \begin{pmatrix} a+bi \\ c+di \end{pmatrix} &\equiv_{\Z[i]} \begin{pmatrix} a+bi \\ c+di \end{pmatrix}.
		\end{align*}
	From these relations we have that in the elliptic curve \[i(c+di)\equiv c+di\quad\textrm{ and }\quad-i(a+bi)+(i-1)(c+di)\equiv a+bi.\]
From the first identity follows $(i-1)(c+di)\equiv 0\in E$ and hence $-i(a+bi)\equiv a+bi$. Then, $(c+d)+(d-c)i\equiv (a-b)+(a+b)i\equiv 0 \in E$. This implies that $a+b,a-b,c+d,d-c\in\Z$, thus $2a,2b,2c,2d\in \Z$. Then, $a,b,c,d\in\left\{0,\pm\dfrac{1}{2}\right\}$, i.e.,
			\[P_{G}^{G}\subset\left\langle\left(\dfrac{1}{2}+\frac{i}{2},0\right),\left( 0, \frac{1}{2}+\frac{i}{2} \right)\right\rangle.\]
	Since the generators are solutions of the system of equations, the inclusion is actually an equality, and therefore $P_{G}^{G}$ is isomorphic to the Klein group. If $\Delta$ is isomorphic to any non trivial subgroup of $P_{G}^{G}$ then the order of $\Delta$ is $2$ or $4$. We know that the action of $\Delta$ on $\P^{2}$ induced by the action of $\Delta$ on $P_{G}$ is faithful. Thus, we have to determine all the faithful representations of $\Delta$ in $\PGL_{3}(\C)$. Let $\Gamma$ be any representation of $\Delta$ in $\PGL_{3}(\C)$ and $H\in \SL_{3}(\C)$ be any lifting of $\Gamma$ by $\pi:\SL_{3}(\C)\to \PGL_{3}(\C)$. We know that the order of $H$ is $3\lvert \Gamma\rvert$. Let us take $H'$ a $2$-Sylow of $H$. Since $H'\cap \ker(\pi)=\{id\}$, it follows that $H'\cong \Gamma$. Then, we can study the representations of $\Delta$ in $\PGL_{3}(\C)$ by studying the representations of $\Delta$ in $\SL_{3}(\C)$. Since $\Delta$ is an abelian group, of order 2 or 4 and exponent 2, there is only one conjugacy class of representations of $\Delta$ in $\SL_{3}(\C)$, because there are only 3 diagonal matrices of order 2. Hence there is just a single conjugacy class in $\PGL_{3}(\C)$, which corresponds to the one studied in Proposition \ref{clasificaciontipoalpha}, more precisely, in points \ref{alpha2} and \ref{alpha4}. This proves Proposition \ref{clasificaciontipogamma}.
	
	\end{proof}

\bibliographystyle{alpha}
\bibliography{biblio}

\end{document}